\documentclass[12pt]{amsart}
\author{Paul \textsc{Poncet}}
\address{CMAP, \'{E}cole Polytechnique, Route de Saclay, 91128 Palaiseau Cedex, France \\
and INRIA, Saclay--\^{I}le-de-France}
\email{poncet@cmap.polytechnique.fr}

\usepackage{stmaryrd}
\usepackage{amssymb}
\usepackage{amsmath}
\usepackage{yhmath}
\usepackage{calrsfs} 
\usepackage{times} 
\usepackage{ifpdf}
\newcommand{\naturels}{\mathbb N}
\newcommand{\reels}{\mathbb R}

\newcommand{\ddint}{\int^{\scriptscriptstyle\infty}\!\!}
\newcommand{\dint}[1]{\int^{\scriptscriptstyle\infty}_{#1}\!\!}

\newtheorem{theorem}{Theorem}[section]
\newtheorem{corollary}[theorem]{Corollary}
\theoremstyle{definition}
\newtheorem{definition}[theorem]{Definition}
\newtheorem{example}[theorem]{Example}
\newtheorem{hypothesis}[theorem]{Hypothesis}

\begin{document}

\title{Two-valued $\sigma$-maxitive measures \\and Mesiar's hypothesis}
\date{\today}



\begin{abstract}
We reformulate Mesiar's hypothesis [Possibility measures, integration and fuzzy possibility measures, Fuzzy Sets and Systems 92 (1997) 191196], which as such was shown to be untrue by Murofushi [Two-valued possibility measures induced by $\sigma$-finite $\sigma$-additive measures, Fuzzy Sets and Systems 126 (2002) 265268]. We prove that a two-valued $\sigma$-maxitive measure can be induced by a $\sigma$-additive measure under the additional condition that it is $\sigma$-principal.
\end{abstract}

\maketitle

\section{Introduction}

In the sequel, $(E,\mathrsfs{B})$ denotes a measurable space.
Recall that a \textit{$\sigma$-maxitive measure} on $\mathrsfs{B}$ is a map $\tau : \mathrsfs{B} \rightarrow \overline{\reels}_+$ such that 
$\tau(\emptyset) = 0$ and 
$$
\tau\left( \bigcup_{j \in \naturels} B_j \right) = \sup_{j \in \naturels} \tau(B_j),
$$ 
for every sequence $(B_j)_{j\in \naturels}$ of elements of $\mathrsfs{B}$. 
A $\sigma$-maxitive measure $\tau$ on $\mathrsfs{B}$ is \textit{normed} if $\tau(E) = 1$; it is \textit{two-valued} if $\tau(\mathrsfs{B}) = \{ 0, 1 \}$, in which case it is normed. A family $\mathrsfs{C}$ of sets is said to satisfy \textit{CCC} (the \textit{countable chain condition}) if every disjoint subfamily of $\mathrsfs{C}$ is countable. A measure (either $\sigma$-maxitive or $\sigma$-additive) $\mu$ on $\mathrsfs{B}$ is then \textit{CCC} if $\{ B \in \mathrsfs{B} : \mu(B) > 0 \}$ satisfies CCC. 

The following hypothesis was proposed by Mesiar \cite{Mesiar97}.

\begin{hypothesis}
Let $\tau$ be a CCC (``countable chain condition'') two-valued $\sigma$-maxitive measure on $\mathrsfs{B}$. Then $\tau$ is induced by some $\sigma$-finite $\sigma$-additive measure $m$ on $\mathrsfs{B}$, i.e.\ $\tau = \delta_{m}$, where $\delta_m(B) = 1$ if $m(B) > 0$ and $\delta_m(B) = 0$ otherwise.
\end{hypothesis}

Murofushi \cite{Murofushi02} provided a counterexample and hence showed that this hypothesis as such is wrong. He focused on finding a necessary and sufficient condition for such a $\tau$ to be induced by some $\sigma$-finite $\sigma$-additive measure. In this article we propose to give up the constraint 
``$\sigma$-finite'' and to replace 
it by a more adequate condition, namely the \textit{$\sigma$-principality} property. 

\begin{definition}\label{loc}
A measure (either $\sigma$-maxitive or $\sigma$-additive) $\mu$ on $\mathrsfs{B}$ is \textit{$\sigma$-principal} if for each $\sigma$-ideal $\mathrsfs{I}$ of $\mathrsfs{B}$, there exists some $L \in \mathrsfs{I}$ such that $\mu(S \backslash L) = 0$ for all $S \in \mathrsfs{I}$. 
\end{definition}

\begin{example}
Let $E$ be a set endowed with its power set ($\mathrsfs{B} = 2^E$). The $\sigma$-additive measure $\# : B \mapsto \# B$, where $\# B$ is the cardinality of $B$, is $\sigma$-principal if and only if $E$ is countable.
\end{example}

A $\sigma$-principal measure is always CCC. See 
Sugeno and Murofushi \cite{Sugeno87} for a proof of the converse statement using Zorn's Lemma. 
Note also that every finite (or $\sigma$-finite) $\sigma$-additive measure is $\sigma$-principal.

\section{Modified Mesiar's Hypothesis}

Mesiar \cite{Mesiar97} noted that, if his hypothesis were true, then every CCC $\sigma$-maxitive measure could be represented as an essential supremum with respect to a $\sigma$-additive measure. We show first that such a representation holds, then prove a modified version of Mesiar's hypothesis.

\begin{theorem}\label{corresssup}
Any $\sigma$-principal (resp.\ CCC) $\sigma$-maxitive measure can be expressed as an essential supremum with respect to a $\sigma$-principal (resp.\ CCC) $\sigma$-additive measure.
\end{theorem}

\begin{proof}
Let $\tau$ be a $\sigma$-principal $\sigma$-maxitive measure on $\mathrsfs{B}$ and $m = \overline{\tau}$ be the map defined on $\mathrsfs{B}$ by
$$
m(B) = \sup_{\pi} \sum_{B' \in \pi} \tau(B\cap B'),
$$
where the supremum is taken over the set of finite $\mathrsfs{B}$-partitions $\pi$ of $E$. It is not difficult to show that $m$, called the \textit{disjoint variation} of $\tau$, is the least $\sigma$-additive measure greater than $\tau$ (see e.g.\ \cite[Theorem~3.2]{Pap95}). Let us show that $m$ is $\sigma$-principal. If $\mathrsfs{I}$ is a $\sigma$-ideal of $\mathrsfs{B}$, there exists some $L \in \mathrsfs{I}$ such that $\tau(B \backslash L) = 0$ for all $B \in \mathrsfs{I}$ (because $\tau$ is $\sigma$-principal). If $B \in \mathrsfs{I}$, then $\tau(B\cap B' \backslash L) = 0$ for all $B' \in \mathrsfs{B}$, since $B\cap B' \in \mathrsfs{I}$. Hence we have $m(B \backslash L) = 0$. Moreover, $m(B) > 0$ implies $\tau(B) > 0$, so that $m$ is CCC if  $\tau$ is CCC.
With the Sugeno--Murofushi theorem (see a reminder in the appendix, Theorem~\ref{sugeno-murofushi}) and the fact that $\tau$ is absolutely continuous with respect to $\delta_m$, one can write $\tau(B) = \ddint c\, d\delta_m = m$-$\sup_{x\in B} c(x)$, where $c : E \rightarrow \overline{\reels}_+$ is a $\mathrsfs{B}$-measurable map.  
\end{proof}

\begin{corollary}
A two-valued $\sigma$-maxitive measure is $\sigma$-principal (resp.\ CCC) if and only if it is induced by a $\sigma$-principal (resp.\ CCC) $\sigma$-additive measure.
\end{corollary}

\begin{proof}
Let $\tau$ be a $\sigma$-principal $\sigma$-maxitive measure on $\mathrsfs{B}$. 
The above construction of $m$ shows that $\tau(B) > 0 \Leftrightarrow m(B) > 0$, which implies that $\tau = \delta_m$ if $\tau$ is two-valued.
\end{proof}


\appendix

\section{}

Sugeno and Murofushi \cite{Sugeno87} proved a Radon--Nikodym like theorem for the Shilkret integral in the case where the dominating $\sigma$-maxitive measure is CCC. Actually their proof remains valid if one just assumes $\sigma$-principality (this is straightforward since they showed under Zorn's Lemma that every CCC measure is $\sigma$-principal). 

\begin{theorem}[Sugeno--Murofushi]\label{sugeno-murofushi}
Let $\tau$, $\nu$ be $\sigma$-maxitive measures on $\mathrsfs{B}$. Assume that $\nu$ is $\sigma$-finite and $\sigma$-principal. The following are equivalent:
\begin{enumerate}
	\item\label{sm1} $\tau$ is absolutely continuous with respect to $\nu$, i.e.\ $\nu(B) = 0$ implies $\tau(B) = 0$, for all $B \in \mathrsfs{B}$, 
	\item\label{sm2} there exists some $\mathrsfs{B}$-measurable map $c : E \rightarrow \overline{\reels}_+$ such that, for all $B\in \mathrsfs{B}$,
\begin{align*}
\tau(B) = \dint{B} c\, d\nu. 
\end{align*}
\end{enumerate}
If (\ref{sm1}) or (\ref{sm2}) holds, then $c$ is unique $\nu$-almost everywhere.
\end{theorem}

Here $\dint{B} c\, d\nu := \sup_{t \in \reels_+} t.\nu(B \cap \{ c > t \})$ denotes the Shilkret integral, see Shilkret \cite{Shilkret71}; see also Poncet  \cite[Chapter~I]{Poncet11}. 
The superscript $\infty$ in the notation of the Shilkret integral finds its justification in Gerritse \cite{Gerritse96}, who stated that the Shilkret integral can be viewed as a limit of Choquet integrals. 

\bibliographystyle{plain}

	

\end{document}